\documentclass[12pt,reqno]{amsart}
\usepackage{amssymb}
\usepackage{mathrsfs}
\usepackage[all]{xy}

\theoremstyle{plain}
\newtheorem{prop}{Proposition}

\newtheorem{theo}[prop]{Theorem}
\newtheorem{coro}[prop]{Corollary}

\theoremstyle{remark}
\newtheorem{rema}[prop]{Remark}
\theoremstyle{definition}

\numberwithin{equation}{section}

\newenvironment{dedication}{\thispagestyle{empty}\itshape\raggedleft}{\par}

\newcommand{\A}{{\mathbb A}}
\newcommand{\C}{{\mathbb C}}
\newcommand{\PP}{{\mathbb P}}

\newcommand{\F}{{\mathbb F}}
\newcommand{\G}{{\mathbb G}}

\newcommand{\R}{{\mathbb R}}
\newcommand{\Z}{{\mathbb Z}}

\newcommand{\mo}{{\mathfrak o}}

\newcommand{\Val}{{\mathcal V}}

\newcommand{\eqto}{\stackrel{\lower1.5pt\hbox{$\scriptstyle\sim\,$}}\to}
\newcommand{\eqdashto}{\stackrel{\lower1.5pt\hbox{$\scriptstyle\sim\,$}}\dashrightarrow}
\DeclareMathOperator{\Gal}{Gal}

\DeclareMathOperator{\tr}{tr}

\DeclareMathOperator{\Pic}{Pic}
\DeclareMathOperator{\Spec}{Spec}

\DeclareMathOperator{\Hom}{Hom}

\DeclareMathOperator{\Br}{Br}

\DeclareMathOperator{\cores}{cores}

\begin{document}
\title{Brauer groups of involution surface bundles}
\author{Andrew Kresch}
\address{
  Institut f\"ur Mathematik,
  Universit\"at Z\"urich,
  Winterthurerstrasse 190,
  CH-8057 Z\"urich, Switzerland
}
\email{andrew.kresch@math.uzh.ch}
\author{Yuri Tschinkel}
\address{
  Courant Institute,
  251 Mercer Street,
  New York, NY 10012, USA
}
\address{Simons Foundation, 160 Fifth Av., New York, NY 10010, USA}
\email{tschinkel@cims.nyu.edu}

\date{September 7, 2018} 

\maketitle

\begin{dedication}
To David Mumford, with admiration.
\end{dedication}

\section{Introduction}
\label{sec:introduction}

A fundamental breakthrough in the study of rationality properties of
complex algebraic varieties was the construction, by Artin and Mumford, of 
examples of projective unirational threefolds with nontrivial Brauer group.
Along with the examples by Iskovskikh--Manin and Clemens--Griffiths, these provided the first instances of nonrational 
unirational complex threefolds, settling the long-standing L\"uroth problem. 
Even more important was the introduction of new tools and concepts:
\begin{itemize}
\item Brauer groups \cite{AM},
\item Birational rigidity  \cite{MI}, and
\item Intermediate Jacobians  \cite{CG}.
\end{itemize}
All of these have triggered major developments in algebraic geometry,  
see e.g.,  \cite{manin-ts}, \cite{beau-survey}, \cite{pukh-book}, \cite{cheltsov-survey},  and the references therein. 

The closely related Zariski problem concerns stable rationality, i.e., 
rationality of the product of the variety in question with some projective space.
The Artin--Mumford examples are not stably rational, while there exist 
threefolds with a nontrivial intermediate Jacobian obstruction to rationality
and which are nevertheless stably rational \cite{BCTSS}.
It is currently unknown whether or not birational rigidity obstructs
stable rationality.

Recent years have seen a tremendous revival of interest in the Artin--Mumford construction
in connection with the Specialization method, introduced by Voisin \cite{voisin}, 
and developed by Colliot-Th\'el\`ene--Pirutka \cite{CTP},
Nicaise--Shinder \cite{NSh}, and Kontsevich--Tschinkel \cite{KT}.
These new techniques relate the failure of
(stable) rationality of a very general member of a family to the
presence of a Brauer group obstruction in a single member of the family.
Often the general members of the family possess no
evident obstructions to rationality, while the
(mildly singular) special member is of Artin--Mumford type.
This led to tremendous advances in the study of stable rationality, see, e.g., 
 \cite{ct-pir-cyclic}, \cite{totaro-JAMS}, \cite{beauvillesexticdouble}, and the surveys \cite{Voisin-survey},  \cite{peyre-ast}. 
In particular, this allowed to: 
\begin{itemize}
\item settle the long-standing open problem of  stable rationality  for rationally connected threefolds, showing that
very general (in their families) nonrational threefolds are not stably rational  
(with the exception of cubic threefolds) \cite{HKTconic},  \cite{HTfano}, \cite{krylovokada};
\item understand the behavior of rationality under deformations, by 
exhibiting smooth families of
complex three- and fourfolds with varying (stable) rationality properties \cite{HKTthreefolds}, \cite{HPT}.
\end{itemize}

These developments focused the attention on varieties with nontrivial Brauer group and mild singularities arising
in interesting families of rationally connected varieties, e.g.,
conic and higher-dimensional quadric bundles over projective spaces \cite{abp},  \cite{schreieder2}.
The computation of the Brauer group on such varieties is an
interesting problem by itself, studied, e.g., by Colliot-Th\'el\`ene--Ojanguren in \cite{CTO} and by Colliot-Th\'el\`ene, 
in the case of conic bundles over rational surfaces  \cite[Thm. 3.13]{pirutkaunramified}.
More recently, Pirutka gave an explicit combinatorial algorithm for the computation of
the Brauer group of quadric surface bundles over rational surfaces \cite{pirutkaunramified}.
It became a crucial ingredient in proofs of failure of stable rationality in \cite{HPT}, \cite{HPT-quadric}, \cite{HPTdouble}, \cite{schreieder1}. 

In these investigations it was important to construct good, i.e., mildly singular, birational models
of varieties fibered over rational surfaces. Already the case of conic and Brauer--Severi surface 
bundles is quite involved \cite{KTsurf}.
In \cite{KTinvsurf}, we studied quadric surface bundles and, more generally,
involution surfaces bundles, with special attention to producing and deforming such models.
In this paper, we use these models to give a combinatorial algorithm
for the computation of the Brauer group (Theorem \ref{thm.rationalsurfacecase}), generalizing
Pirutka's algorithm. The inspiration comes from the work of Artin and Mumford in the
conic bundle case. 

\

\noindent {\bf Acknowledgments:}  We are very grateful to 
Brendan Hassett and Alena Pirutka 
for discussions on these and related topics. The second author
was partially supported by NSF grant 1601912. This work was done while the second author 
was visiting the FIM, at ETH Zurich. Its hospitality is greatly appreciated.

\section{Geometry of involution surface bundles}
\label{sec:geometry}

Let $K$ be a field of characteristic different from $2$.
A surface over $K$ that is
geometrically isomorphic to $\PP^1\times \PP^1$ is called an
\emph{involution surface}.
Such a surface is classified by the pair $(L/K,\beta)$, where 
\begin{itemize}
\item $L/K$ is the 
\emph{discriminant extension}, a degree $2$
\'etale $K$-algebra, and
\item $\beta$ is a $2$-torsion element of the
Brauer group of $L$ that is the class of a quaternion algebra.
\end{itemize}
We remark that
for the unique nontrivial $K$-automorphism $\tau$ of $L$,
Brauer group elements $\beta$ and $\tau^*(\beta)$ determine
isomorphic involution surfaces.
This ambiguity is eliminated by
fixing a compatible collection, for any
$K$-algebra $\Lambda$, of identifications of the set of
rulings (maps onto a conic) of $X_\Lambda$ with $\Hom_K(\Lambda,L)$.
We assume that such identifications are fixed, without explicit mention,
whenever a degree $2$ \'etale $K$-algebra and a Brauer group element are
mentioned in connection with an involution surface.

We work over an algebraically closed ground field $k$
of characteristic different from $2$ and let $S$ be a nonsingular
algebraic variety over $k$.
An \emph{involution surface bundle} over $S$ is a flat generically smooth projective morphism
$$
\pi\colon X\to S,
$$ 
such that if $U\subset S$ denotes the locus over which
$\pi$ is smooth, then the fiber over every point of $U$
is an involution surface.
Involution surface bundles were studied in \cite{KTinvsurf}, where we identified 
four geometric types of degenerations of involution surfaces, Types I, II, III, and IV.
Involution surface bundles with only these geometric types of degenerations, and
satisfying further conditions restricting the
singularities of the total space $X$, 
were called \emph{mildly degenerating simple} involution surface bundles.

From now on we suppose that $S$ is a smooth projective surface over $k$.
Good models will be mildly degenerating simple involution surface bundles
over the complement of a codimension $2$ set $Z\subset S$ (finitely many points),
with additional degeneration types permitted at points of $Z$.
Specifically, an involution surface bundle over $S$ is defined to be
\emph{simple} if the complement of $U$ is a simple normal crossing divisor
\[ D=D_1\cup D_2\cup D_3\cup D_4, \]
where we take $Z=D^{\mathrm{sing}}$.
The geometric fibers should have Type I over $D_1$, Type II over $D_2$,
Type III over $D_3$, and Type IV over $D_4$.
Additionally, $D_1$ and $D_3$ are required to be smooth, disjoint from each
other, and disjoint from $D_4$, and $\pi$ is required to have
one of $6$ explicit isomorphism types \'etale locally
at every point of $Z$.

By \cite[Thm.\ 10]{KTinvsurf}, any fibration $\pi\colon X\to S$ whose generic fiber is
an involution surface admits a model
$\tilde\pi\colon \widetilde{X}\to \widetilde{S}$ over some smooth surface $\widetilde{S}$ with
proper birational morphism to $S$, which is a simple involution surface bundle.
The proof translates into the following recipe.
Let $K$ denote the function field of $S$, and $L$, the discriminant extension
of the generic fiber.
If $L$ is a quadratic field extension, then from a model we obtain, by
birational modification, a finite degree $2$ morphism of
smooth surfaces $\widetilde{T}\to \widetilde{S}$, while
in case $L\cong K\times K$
we take $\widetilde{S}=S$ and $\widetilde{T}=S\sqcup S$.
Let $\beta$ denote the $2$-torsion element of $L$, corresponding to the
generic fiber of $\pi$, represented geometrically by a conic bundle
over $\widetilde{T}$.
This is put into a standard form after
further blowing up $\widetilde{T}$, which may be done compatibly with
blow-ups of $\widetilde{S}$.
The standard conic bundle determines, by \cite[Thm.\ 13]{KTinvsurf},
a simple involution surface bundle, with following data:
\begin{itemize}
\item The branch locus of $\widetilde{T}\to \widetilde{S}$ is $D_1\cup D_3$,
where over $D_1$ the element $\beta$ is unramified and the
conic bundle has smooth fibers, and over $D_3$ the Brauer group element is
ramified and the conic bundle has, generically, reduced singlar fibers.
\item The additional divisors where $\beta$ is ramified
lie over $D_2$ and $D_4$.
\end{itemize}

Every generically smooth quadric surface bundle determines
a simple involution surface bundle with $D_3=D_4=\emptyset$.
However, there exist simple involution surface bundles with
$D_3=D_4=\emptyset$ which are not models of quadric surface bundles.
A necessary and sufficient condition to be a model of a quadric surface bundle
is that $\beta$ lies in the kernel of the corestriction homomorphism
$\Br(L)[2]\to \Br(K)[2]$, or equivalently lies in the image of the
restriction homomorphism $\Br(K)[2]\to \Br(L)[2]$.

\section{Resolution}
\label{sec.resolution}
The hypersurface singularity defined by $uv=xyz$ has singular locus
consisting of the union of three curves and may be resolved by
blowing up first one of the curves, then the proper transforms of the other two.
These assertions are straightforward to verify (over an arbitrary field
$k$), either
by recognizing $uv=xyz$ as defining the affine toric variety given by the cone
\[ \R_{\ge 0}\langle
(1,0,0,0), (0,1,0,0), (1,0,1,0), (0,1,1,0), (1,0,0,1), (0,1,0,1)
\rangle \]
in $N\otimes_\Z\R$, where $N=\Z^4$, or by computing the blow-ups directly
in local coordinates.
For the first blow-up,
the fiber over the origin is the union of two copies of $\PP^2$ along
a line, and smooth quadric surfaces over all other points of the
center of the blow-up.
All the fibers of the second blow-up (over the center of blow-up) are
smooth quadric surfaces.

We recall that a proper morphism of finite-type schemes over a field $k$
is said to be
universally $CH_0$-trivial if the induced push-forward
morphism on the groups $CH_0$ of zero-cycles up to rational equivalence
is an isomorphism, not only over the given field but also after base-change
to an arbitrary extension field;
a proper scheme is said to be
universally $CH_0$-trivial if the structure morphism to $\Spec(k)$ is.
A sufficient condition for a proper morphism to be
universally $CH_0$-trivial is that its fibers over all points
(closed or not) are universally $CH_0$-trivial \cite[Prop.\ 1.8]{CTP}.

\begin{theo}
\label{thm.resolution}
Let $k$ be an algebraically closed field of characteristic different from $2$,
$S$ a smooth projective surface over $k$, and 
$\pi\colon X\to S$ a simple involution surface bundle, with singular fibers
over
\[ D = D_1\cup D_2\cup D_3\cup D_4. \]
Then by
\begin{itemize}
\item blowing up $X$ along the copy of the normalization of
$D_2$, which is the closure in $X$ of the
singular locus of $\pi^{-1}(S\smallsetminus D^{\mathrm{sing}})$:
\[ \varphi\colon X'\to X, \]
\item blowing up $X'$ along its singular locus, which consists of two
disjoint curves in the fiber over every point of $D_2^{\mathrm{sing}}$:
\[ \varphi'\colon X''\to X', \]
\end{itemize}
we obtain $\varphi\circ \varphi'$,
a universally $CH_0$-trivial desingularization $X''\to X$.
\end{theo}

\begin{proof}
As indicated in \cite[Defn.\ 5]{KTinvsurf}, the only singularities of $X$
over the complement of $D^{\mathrm{sing}}$ are
double point singularities
along the section over $D_2\smallsetminus D^{\mathrm{sing}}$.
Over $D^{\mathrm{sing}}$ the singularities of $X$
are described in \cite[\S 3]{KTinvsurf}, and by these descriptions,
the closure in $X$ of the indicated section is isomorphic to the
normalization of $D_2$.
We recall the description.
\begin{itemize}
\item I meets II: $X$ has only ordinary double point singularities along
the copy of $D_2$.
\item II meets II: over a Zariski neighborhood of a point
$z\in D_2^{\mathrm{sing}}$,
\begin{itemize}
\item the copies in $X$ of the two components of $D_2$ containing $z$
intersect $\pi^{-1}(z)$ at distinct points $z'$ and $z''$;
\item $X$ has
ordinary double point singularities generically along the two components
and two additional curves in $\pi^{-1}(z)$, both
containing $z'$ and $z''$;
\item the \'etale local isomorphism type of the singularities of $X$ at
$z'$ and $z''$ is that of the hypersurface singularity $uv=xyz$.
\end{itemize}
\item III meets II and IV meets II: $X$ has singularity of type $\mathsf{D}_\infty$ along the copy of $D_2$.
\end{itemize}
Ordinary double point singularities along a curve in $X$ are resolved by
blowing up the curve.
As indicated in \cite{zaharia}, the same holds for singularities of type
$\mathsf{D}_\infty$.
Over a neighborhood of a singular point of $D_2$, blowing up the copy of the
normalization of $D_2$ amounts to the first step of the indicated resolution
of the hypersurface singularity $uv=xyz$, and the remaining singularities,
ordinary double points along curves, are resolved by blowing up those curves.
With each blow-up, the fiber over any point is a
union of two copies of $\PP^2$ along $\PP^1$ over $k$,
a nodal quadric surface over $k$, or a nonsingular quadric surface over
$k$ or over the function field of a curve over $k$.
Each of these is universally $CH_0$-trivial.
\end{proof}

\section{Brauer group computation}
\label{sec:brauergroup}
Let $Y$ be an algebraic variety over $k$.
For an extension field $F/k$ we let $Y(F)$ denote the set of $F$-rational points
and $Y_F$ the base-change to $F$ of $Y$.

Let $L=k(Y)$ be the function field of $Y$.
We let $\Val_L$ denote the set 
of (geometric) divisorial valuations of $L$; in particular any $v\in \Val_L$
is discrete of rank one and is trivial on $k$.
If $Y$ is normal, we let $\Val_Y\subset \Val_L$ denote the subset 
of divisorial valuations whose centers on $Y$ are irreducible divisors. 
For $v\in \Val_L$
we write $\mo_v$ for the corresponding local ring and
$\kappa_v$ for the residue field.
We denote the henselization by $\mo^h_v$ and its field of fractions
by $K^h_v$.

For a positive integer $\ell$ invertible in $k$ we fix
an isomorphism $\mu_{\ell}\simeq \Z/\ell\Z$.
We write 
$$
H^i(Y):=H^i_{et}(Y,\Z/\ell\Z), \qquad  H^i(L):=H^i(\Spec(L)),
$$
when the coefficients are clear from the context. 
For every $v\in \Val_L$ we have residue homomorphisms
$$
H^i(L)\stackrel{\partial_v}{\longrightarrow} H^{i-1}(\kappa_v).
$$
The unramified cohomology of $Y$ is an invariant of its
function field $L$ as an extension of $k$, defined by 
$$
H^i_{nr}(L/k):=\bigcap_{v\in \Val_L} \ker(\partial_v), 
$$
see  \cite{CTO}.
When the base field is clear from context, we will write $H^i_{nr}(L)$.
When $Y$ is smooth and projective, we also have
$$
H^i_{nr}(L) = \bigcap_{v\in \Val_Y}\ker(\partial_v),
$$
with isomorphisms 
$$
H^1(Y)\cong H^1_{nr}(L)\qquad \text{and} \qquad \Br(Y)[\ell]\cong H^2_{nr}(L).
$$


Let $K$ be a field 
and
$$
G_K=\Gal(\overline{K}/K)
$$ 
the Galois group of a separable closure $\overline{K}$ of $K$.

\begin{prop}
\label{prop.surjective}
Let $K$ be a field of characteristic different from $2$ and
$W$ an involution surface over $K$ with
discriminant extension $L/K$ and Brauer group element $\beta\in \Br(L)$.
The restriction map
\[ 
\Br(K)\to \Br(W), 
\]
is surjective, with kernel
\[
\begin{cases}
\langle \cores_{L/K}(\beta)\rangle,&\text{if $L$ is a field},\\
\langle \beta_1,\beta_2\rangle,&\text{if $L\cong K\times K$,
$\beta=(\beta_1,\beta_2)\in \Br(K)\times\Br(K)$}.
\end{cases}
\]
\end{prop}

\begin{proof}
Since $\Br(W_{\overline{K}})=0$,
the Hochschild-Serre spectral sequence
$$
H^p(G_K,H^q(W_{\overline{K}},\G_m))\Rightarrow
H^{p+q}(W,\G_m)
$$
gives rise to the exact sequence
\[
0\to \Pic(W)\to \Pic(W_{\overline{K}})^{G_K}\to
\Br(K)\to \Br(W)\to H^1(G_K,\Pic(W_{\overline{K}})),
\]
The Galois group $G_K$ acts on 
$\Pic(W_{\overline{K}})\cong \Z^2$ via the permutation action
on rulings when $L$ is a field, and trivially when $L\cong K\times K$.
In either case, the Galois cohomology $H^1$ vanishes, and the
surjectivity of the restriction map follows.
The description of the kernel is given in \cite[Prop.\ 5.3]{CTKM}.
\end{proof}

\begin{prop}
\label{prop.kernelres}
Let $k$ be an algebraically closed field of characteristic different from $2$,
$S$ a nonsingular algebraic variety over $k$, and 
$$
\pi\colon X\to S
$$
a mildly degenerating simple involution surface bundle,
smooth over $U\subset S$, with degenerate fibers over
\[ D=D_1\cup D_2\cup D_3\cup D_4. \]
With $K=k(S)$ and $W=X_K$ we adopt the further notation of
Proposition \ref{prop.surjective}.
For $v\in \Val_S$ and coefficients $\Z/\ell\Z$, where $\ell$ is a
positive integer, invertible in $k$, the restriction map
\begin{equation}
\label{eqn:rho}  
\rho_v\colon
H^1(\kappa_v)\to \bigoplus_{\substack{w\in \Val_X\\ w|_K=v}}H^1(\kappa_w) 
\end{equation}
is injective when $\ell$ is odd and has kernel
\[
\begin{cases}
0,& \text{if $v\in U$ or $v\in D_1$}, \\
\langle \partial_v(\cores_{L/K}(\beta)) \rangle,& \text{if $v\in D_2$ and
$v$ is inert in $L$}, \\
\langle \partial_{v_1}(\beta),\partial_{v_2}(\beta) \rangle,& \text{if $v\in D_2$, extending
to distinct $v_1$, $v_2\in \Val_L$}, \\
\langle \partial_{v'}(\beta)\rangle, & \text{if $v\in D_3$ with unique extension $v'\in \Val_L$}, \\
\langle \partial_{\varepsilon}(\beta)\rangle,& \text{if $v\in D_4$, marked by
$\varepsilon$ over $v$},
\end{cases}
\]
when $\ell$ is even.
\end{prop}

By abuse of notation, in case $L\cong K\times K$ we consider $v$ as
split in $L$, with $\partial_{v_i}(\beta)=\partial_v(\beta_i)$ for
$i=1$, $2$.
Every component of a Type IV divisor has a \emph{marking}
$\varepsilon\in \Val_L$ extending $v$, with the property
that $\beta$ extends to an element of the Brauer group of
$\Spec(\mo'_v)\smallsetminus \{\varepsilon\}$, where
$\mo'_v$ denotes the integral closure of $\mo_v$ in $L$.
We will employ the notation $\Val_T$ analogously when
$T\cong S\sqcup S$.

\begin{proof}
We proceed via a case-by-case analysis: 
\begin{itemize}
\item 
$v\in U$ or $v\in D_1$.  Then
$X_{\kappa_v}$ is geometrically integral, hence the kernel is trivial.
\item 
$v\in D_2$ and $v$ is inert, extending uniquely to
a valuation $v'$ on $L$.
Then, according to the description of
mildly degenerating simple involution surface bundles from \cite{KTinvsurf},
$X_{\kappa_v}$ is the restriction of scalars under $\kappa_{v'}/\kappa_v$
of the singular conic in $\PP^2_{\kappa_{v'}}$, defined by an equation of the
form
\[ X^2-rY^2=0, \]
where $r\in \kappa_{v'}^\times$ is a representative of
$$
\partial_{v'}(\beta)\in H^1(\kappa_{v'},\Z/2\Z)\cong
\kappa_{v'}^\times/\kappa_{v'}^{\times2}.
$$
The conic contains a dense open subscheme isomorphic to
$\A^1_{\kappa_{v'}(\sqrt{r})}$, hence a dense open subscheme of $X_{\kappa_v}$
is isomorphic to $\A^2_{\Lambda}$, where $\Lambda$ is the coordinate ring of
the restriction of scalars under $\kappa_{v'}/\kappa_v$ of
$\Spec(\kappa_{v'}(\sqrt{r}))$.
\begin{itemize}
\item 
If $r\in \kappa_v$, then 
$$
\Lambda\cong \kappa_v(\sqrt{r})\times \kappa_v(\sqrt{rs}),
$$
where $s\in \kappa_v$ is such that $\kappa_{v'}\cong \kappa_v(\sqrt{s})$.
\item 
If $r\notin \kappa_v$ but
$N_{\kappa_{v'}/\kappa_v}(r)=c^2\in \kappa_v^{\times2}$, then
\[
\textstyle
\Lambda\cong \kappa\big(\sqrt{\tr_{\kappa_{v'}/\kappa_v}(r)+2c}\big)\times
\kappa\big(\sqrt{\tr_{\kappa_{v'}/\kappa_v}(r)-2c}\big),
\]
where we observe that
$$
(\tr_{\kappa_{v'}/\kappa_v}(r)+2c)(\tr_{\kappa_{v'}/\kappa_v}(r)-2c)
$$ 
is
equal to a square times $s$.
\end{itemize}
So, in these two cases, the kernel is trivial.
(The kernel is also trivial when $\partial_{v'}(\beta)=0$,
$\Lambda\cong \kappa_v\times\kappa_v\times\kappa_{v'}$.)
\begin{itemize}
\item 
If $r\notin \kappa_v$ and $N_{\kappa_{v'}/\kappa_v}(r)\notin \kappa_v^{\times2}$, then
$\Lambda$ is a quadratic extension of $\kappa_v(\sqrt{N_{\kappa_{v'}/\kappa_v}(r)})$ which
as extension of $\kappa_v$ is either cyclic or non-Galois.
So the kernel (when $\ell$ is even) is
$$
\langle N_{\kappa_{v'}/\kappa_v}(r)\rangle=
\langle \partial_v(\cores_{L/K}(\beta)) \rangle.
$$
\end{itemize}
\item
$v\in D_2$ and $v$ is split. Then
$X_{\kappa_v}$ is a product of singular conics.
We leave details of this case to the reader.
\item 
$v\in D_3$. Then the construction of mildly degenerating
involution surface bundles given in \cite[Thm.\ 6]{KTinvsurf}
(out of a conic bundle corresponding to the Brauer class $\beta$)
leads to a description of a dense open subscheme of
$X_{\kappa_v}$ as $\A^2_{\kappa_v(\sqrt{r})}$ where
$r\in \kappa_v^\times=\kappa_{v'}^\times$ is a representative of
$\partial_{v'}(\beta)$
(or two copies of $\A^2_{\kappa_v}$ when
$\partial_{v'}(\beta)=0$).
\item 
$v\in D_4$. Then
$X_{\kappa(v)}$ is a product of a singular conic and a
nonsingular conic,
and the kernel is as claimed. \qedhere
\end{itemize}
\end{proof}

Suppose, now, $S$ is a nonsingular projective surface over $k$,
and $X$ is a simple involution surface bundle over $S$.
We are interested in knowing when $\alpha\in \Br(K)[\ell]$
(with $\ell$ invertible in $k$)
restricts under $\pi$ to an element of $\Br(X_K)[\ell]$ that
is unramified, i.e., has trivial residue for all valuations in
$\Val_{k(X)}$, or equivalently, for all valuations in
$\Val_{\widetilde{X}}$, where $\widetilde{X}$ is a desingularization of $X$.
By Proposition \ref{prop.kernelres}, a \emph{necessary} condition for
this is that $\partial_v(\alpha)$ should
belong to the kernel of the map $\rho_v$ in \eqref{eqn:rho}, for all $v\in \Val_S$.
Indeed, there is a commutative diagram
\[
\xymatrix{
\Br(k(X))[\ell] \ar[r]^(0.55){\partial_w} & H^1(\kappa_w) \\
\Br(K)[\ell] \ar[r]^(0.55){\partial_v} \ar[u]^{\rho} & H^1(\kappa_v) \ar[u]
}
\]
as in \cite[\S 1]{CTO}, where the vertical maps are restriction maps and
the factor coming from the valuation under $w$ of a uniformizer of $v$
is always $1$, since $\pi$ is smooth outside of a locus of
codimension at least $2$.
The next result shows that this condition is also sufficient.

\begin{prop}
\label{prop.sufficient}
Let $k$ be an algebraically closed field of characteristic different from $2$,
$S$ a nonsingular surface over $k$ with function field $K$, and 
$\pi\colon X\to S$ a simple involution surface bundle. Let
$\alpha\in \Br(K)[\ell]$, with $\ell$ invertible in $k$, be an element 
such that 
$$
\partial_v(\alpha) \in \mathrm{ker}(\rho_v), \quad \text{ for all } v\in \Val_S.
$$
Then, for every $v\in \Val_S$ with $\partial_v(\alpha)\ne 0$, we have
\[ \alpha\in \ker(\Br(K)[\ell]\to \Br(X_{K^h_v})[\ell]), \]
and, for every $z\in D^{\mathrm{sing}}$, we have
\[ \alpha\in \ker(\Br(K)[\ell]\to \Br(X_{K^h_z})[\ell]), \]
where $K^h_z$ denote the fraction field of the henselization
$\mo^h_z$ of the local ring at $z$.
\end{prop}

\begin{proof}
Suppose first that $v$ is in $D_2$ and is inert in $L$.
Then $\ell$ is even, and
$\partial_v(\alpha)=\partial_v(\cores_{L/K}(\beta))$.
Since 
$$
\cores_{L/K}(\beta))\in \ker(\Br(K)\to \Br(X_K)),
$$ 
it
suffices to show that
\[ \alpha-\cores_{L/K}(\beta)\in \ker(\Br(K)[\ell]\to \Br(K^h_v)[\ell]). \]
By the equality of residues,
$\alpha-\cores_{L/K}(\beta)$ is the restriction of an element of
$\Br(\mo^h_v)$.
But $\Br(\mo^h_v)=0$, since $\kappa_v$ is a $C_1$-field (Tsen's theorem).
The same argument takes care of the cases $v\in D_3$ and $v\in D_4$.

It remains to treat the case that $v$ in $D_2$ is split in $L$,
and the case of $z\in D^{\mathrm{sing}}$.
In this case, we have $L\otimes_KK^h_v\cong K^h_v\times K^h_v$.
We are then reduced to the case $L\cong K\times K$, and we may argue as above,
using
$$
\partial_v(\alpha)\in \{\partial_v(\beta_1),\partial_v(\beta_2),
\partial_v(\beta_1+\beta_2)\}.
$$
For $z\in D^{\mathrm{sing}}$, either $\alpha$ vanishes on an
\'etale neighborhood of $z$, in which case the assertion is trivial, or else
after passing to a suitable \'etale neighborhood we have
$\alpha=(x,y)$ where $x$ and $y$ are local defining equations of the
components of $D$ containing $z$.
We divide into subcases according to the \'etale local isomorphism type of $X$,
using the notation from \cite{KTinvsurf} and noting that cases
$\widehat{X}_{I,II}$ and $\widehat{X}''_{IV,IV}$ are trivial for the
above reason.
In all of the remaining cases, except
$\widehat{X}_{III,II}$, we may assume $L\cong K\times K$ and obtain, by
Proposition \ref{prop.kernelres} (applied after base change to a suitable
\'etale neighborhood) kernel generated by $(x,y)$.
In case $\widehat{X}_{III,II}$ we adopt the notation of \cite[\S 3.3]{KTinvsurf}:
$x$ is a local defining equation of $D_3$, $y$ of $D_2$, and
$L=K(s)$ where $s^2=x$.
Now \cite[18.8.10]{EGAIV} is applicable to the branched degree $2$ covering
of $S$ and tells us that the \'etale local form $(s,y)$
of $\beta\in \Br(L)$ (which we have by the same argument as above) is
achieved after passing to a suitable \'etale neighborhood of $z$ in $S$.
This corestricts to $(x,y)$, and we conclude as before.
\end{proof}

\begin{coro}
\label{cor.sufficient}
Let $k$ be an algebraically closed field of characteristic different from $2$,
$S$ a nonsingular projective surface over $k$ with function field $K$, and
$\pi\colon X\to S$ be a simple involution surface bundle. Let 
$\alpha\in \Br(K)[\ell]$, with $\ell$ invertible in $k$.
Then the following are equivalent:
\begin{itemize}
\item[(i)] $\partial_v(\alpha)\in \ker(\rho_v)$, for every $v\in \Val_S$;
\item[(ii)] $\rho(\alpha)\in H^2_{nr}(k(X)/k)$, i.e., if
$\widetilde{X}$ denotes any desingularization of $X$ then
$\rho(\alpha)$ is the restriction of an element of $\Br(\widetilde{X})[\ell]$.
\end{itemize}
\end{coro}

\begin{proof}
By Proposition \ref{prop.kernelres}, (ii) implies (i).
Now suppose (i) is satisfied.
We need to show that for any $w\in \Val_{k(X)}$ we have
$\partial_w(\alpha)=0$.
Since $\rho(\alpha)\in \Br(X_K)[\ell]$, the residue is trivial for all
valuations that restrict to the trivial valuation on $K$.
So we consider only valuations $w$ restricting nontrivially to $K$.

Suppose, first, that $w$ restricts to some $v\in \Val_S$.
By Proposition \ref{prop.sufficient},
there exist an \'etale morphism $S'\to S$ and $v'\in \Val_{S'}$
extending $v$ and inducing an isomorphism on residue fields, such that
\begin{equation}
\label{eqn.inker}
\alpha\in \ker(\Br(K)[\ell]\to \Br(X_{K'})[\ell]),
\end{equation}
where $K'=k(S')$.
Since the residue homomorphism commutes with restriction under an
\'etale morphism, we have $\partial_w(\alpha)=0$.

It remains to consider the case that the restriction of $w$ to $K$ is
centered on some $k$-point $z\in S$.
We will show that there exists a pointed \'etale neighborhood $(S',z')$
of $(S,z)$ for which \eqref{eqn.inker} holds, where $K'=k(S')$.
As before, the vanishing of $\partial_w$ follows.
If $z\notin D^{\mathrm{sing}}$ then $\alpha$ restricts to $0$
in $\Br(K')[\ell]$ for some \'etale neighborhood.
If $z\in D^{\mathrm{sing}}$, then we are done by Proposition \ref{prop.sufficient}.
\end{proof}

In case $S$ is a nonsingular projective \emph{rational} surface,
$\Br(S)=0$ and elements of $\Br(K)[\ell]$ are described completely with
ramification data, according to the exact sequence from \cite[Thm.\ 1]{AM}:
\begin{equation}
\label{eqn.exactsequenceAM}
0\to \Br(K)[\ell]\to \bigoplus_{v\in \Val_S} H^1(\kappa_v)\to
\bigoplus_{z\in S(k)} \Z/\ell\Z.
\end{equation}
In particular, in this setting,
condition (i) in Corollary \ref{cor.sufficient} forces
$\alpha$ to be $2$-torsion in $\Br(K)$.
In light of this, we take $\ell=2$ and work with coefficients $\Z/2\Z$ in the
following concrete description of the unramified Brauer group of an
involution surface bundle over a projective rational surface.

\begin{theo}
\label{thm.rationalsurfacecase}
Let $k$ be an algebraically closed field of characteristic different from $2$,
$S$ a nonsingular projective rational surface over $k$ with
function field $K$, and
$\pi\colon X\to S$ a simple involution surface bundle, such that
the associated conic bundle under the correspondence of
\cite[Thm.\ 13]{KTinvsurf} is a standard conic bundle over a
degree $2$ covering $T\to S$.
Define 
\begin{itemize}
\item 
$L=k(T)$, when $T$ is irreducible, and 
\item $L=K\times K$, when
$T=S\sqcup S$,
\end{itemize}
and let
$\beta\in \Br(L)[2]$ be the class of the standard conic bundle over $T$.
Define
\begin{align*}
\mathcal{S}&=\begin{cases}
0,&\text{if $L$ is a field, $\cores_{L/K}(\beta)=0$},\\
\F_2,&\text{if $L$ is a field, $\cores_{L/K}(\beta)\ne 0$},\\
\bigoplus_{\substack{\beta_i\ne 0\\ i=1\text{ or }\beta_1\ne\beta_2  }} \F_2,&\text{if $L=K\sqcup K$, $\beta=(\beta_1,\beta_2)\in \Br(K)\times \Br(K)$},
\end{cases} \\
\mathcal{P}&=\bigoplus_{v'\in \Val_{\mathcal{P}}} \F_2,\ \Val_{\mathcal{P}}
=\{v'\in \Val_T\,|\,\partial_{v'}(\beta)\ne 0\}, \\
\mathcal{Q}&=\bigoplus_{v\in \Val_{\mathcal{Q}}} \F_2,\ \Val_{\mathcal{Q}}
=\{v\in \Val_S\,|\,\text{$\partial_v(\cores_{L/K}(\beta))=0$, $\exists\,v'\in \Val_{\mathcal{P}}$: $v'|_K=v$}\}, \\
\mathcal{R}&=\bigoplus_{z\in \mathcal{Z}_{\mathcal{R}}} \F_2,\ \mathcal{Z}_{\mathcal{R}}
=\{z\in D^{\mathrm{sing}}\,|\,\text{type
$\widehat{X}_{II,II}$, $\widehat{X}_{III,II}$,
$\widehat{X}_{IV,II}$, or $\widehat{X}'_{IV,IV}$}\},
\end{align*}
where $D$ denotes the simple normal crossing divisor over which
$\pi$ has singular fibers.
We define homomorphisms 
\begin{itemize}
\item 
$\mathcal{S}\to \mathcal{P}$ as the
diagonal inclusion when $L$ is a field, and the product of
diagonal inclusions according to the convention of
$\Val_T$ stated after Proposition \ref{prop.kernelres},
otherwise;
\item 
$\mathcal{Q}\to \mathcal{P}$ by the relation of extension of valuations, and
\item 
$\mathcal{P}\to \mathcal{R}$ by the relation in Table \ref{table1}.
\end{itemize}
Then 
\begin{itemize}
\item 
$\mathcal{S}\to \mathcal{P}$ and $\mathcal{Q}\to \mathcal{P}$ are
injective with trivially intersecting images,
\item 
the composite $\mathcal{Q}\to \mathcal{R}$ is zero,
and 
\item 
the recipe of Table \ref{table2}
identifies
$\mathcal{S}$ with $\ker(\Br(K)\to \Br(X_K))$ and 
$\ker(\mathcal{P}/\mathcal{Q}\to \mathcal{R})$
with the pre-image
of the subgroup $H^2_{nr}(k(X)/k)$.
\end{itemize}
\end{theo}

The elements of $\Val_{\mathcal{P}}$ correspond to the components of
the pre-image of $D_2$, the components of
$D_3$, and the marked components over $D_4$.
The last statement of the theorem
is summarized by the following commutative diagram with exact rows:
\[
\xymatrix{
0 \ar[r] & \mathcal{S} \ar[r] \ar@{=}[d] &
\ker(\mathcal{P}/\mathcal{Q}\to \mathcal{R}) \ar[r] \ar@{^{(}->}[d] & H^2_{nr}(k(X)/k) \ar[r] \ar@{^{(}->}[d] & 0 \\
0 \ar[r] & \mathcal{S} \ar[r] & \Br(K) \ar[r] & \Br(X_K) \ar[r] & 0
}
\]

\begin{table}
\[
\begin{array}{l|c|c|c}
&\text{$v'|_K=v$ Type II}&\text{$v$ Type III}&\text{$v$ Type IV}\\ \hline
\text{$z$ Type $\widehat{X}_{II,II}$} & s & & \\
\text{$z$ Type $\widehat{X}_{III,II}$} & \times & \times & \\
\text{$z$ Type $\widehat{X}_{IV,II}$} & m & & \times \\
\text{$z$ Type $\widehat{X}'_{IV,IV}$} & & & \times
\end{array}
\]
\caption{Relation between $\Val_{\mathcal{P}}$ and $\mathcal{Z}_{\mathcal{R}}$.
For $v'\in \Val_{\mathcal{P}}$, restricting to $v\in \Val_S$ corresponding to
a divisor containing $z\in \mathcal{Z}_{\mathcal{R}}$ the symbol
$\times$ indicates that $v'$ is related to $z$;
$s$ indicates that $v'$ is related to $z$ when $v$ is split in $L$;
$m$ indicates that $v'$ is related to $z$ when
the divisor corresponding to $v'$ meets the marked Type IV component
at a point above $z$.}
\label{table1}
\end{table}

\begin{table}
\[
\begin{array}{l|c|c|c}
&\text{$v'|_K=v$ Type II}&\text{$v$ Type III}&\text{$v$ Type IV} \\ \hline
\text{inert} & \partial_v(\cores_{L/K}(\beta)) & & \\
\text{split} & \partial_{v'}(\beta) & & \partial_{v'}(\beta) \\
\text{ramified} & & \partial_{v'}(\beta) &
\end{array}
\]
\caption{Homomorphism $\mathcal{P}\to \bigoplus_{v\in \Val_{\mathcal{Q}}}H^1(\kappa_v)$
determining $\ker(\mathcal{P}\to \mathcal{R})\to \Br(K)[2]$ by the
representation of an element of $\Br(K)[2]$ by ramification data
in $\bigoplus_{v\in \Val_S}H^1(\kappa_v)$.}
\label{table2}
\end{table}

\begin{rema}
\label{rema.rationalsurfacecase}
We remark that the construction in
\cite[Thm.\ 10]{KTinvsurf} of good models of
involution surface bundles
(i.e., models which are simple involution surface bundles)
proceeds via a standard conic bundle over a degree $2$ covering of
a birational model of the base surface, and hence these particular
simple involution surface bundles satisfy the condition stated in
Theorem \ref{thm.rationalsurfacecase}.
\end{rema}

\begin{proof}[Proof of Theorem \ref{thm.rationalsurfacecase}]
We use the exact sequence \eqref{eqn.exactsequenceAM}, which identifies
$\Br(K)[2]$ with ramification data at divisors satisfying compatibility
conditions at points.
The assertions about $\mathcal{S}\to \mathcal{P}$,
$\mathcal{Q}\to \mathcal{P}$, and
$\mathcal{Q}\to \mathcal{R}$ are readily verified.
By Corollary \ref{cor.sufficient},
the pre-image of $H^2_{nr}(k(X)/k)$ under $\Br(K)\to \Br(X_K)$
consists of elements whose ramification data are
constrained to lie in the kernels described in
Proposition \ref{prop.kernelres}.
The direct sum of these, we check, is
identified with $\mathcal{P}/\mathcal{Q}$ by the
homomorphism described in Table \ref{table2}.
We check, as well, that the homomorphism encoded by
Table \ref{table1} corresponds to the compatibility conditions at
points from \eqref{eqn.exactsequenceAM}.
Finally, $\mathcal{S}$ is identified with the kernel of
$\Br(K)\to \Br(X_K)$ from Proposition \ref{prop.surjective}.
\end{proof}

\section{Example}
\label{sec:example}
Here we demonstrate Theorem \ref{thm.rationalsurfacecase}
on the example from \cite{HPT}:
\[
X:\ \ yzs^2+xzt^2+xyu^2+(x^2+y^2+z^2-2xy-2xz-2yz)v^2=0.
\]
This is a hypersurface in $\PP^3\times \PP^2$,
where the respective factors have homogeneous coordinates
$s$, $t$, $u$, $v$ and $x$, $y$, $z$, and is
a quadric surface bundle over $S=\PP^2$ with
discriminant extension given by the degree $2$ covering branched over
\[ C:\ \ x^2+y^2+z^2-2xy-2xz-2yz=0. \]

The fourfold $X$ appears (modulo birational transformations) as the limit of several interesting families of varieties 
\cite{HPT}, \cite{HPTdouble}, \cite{HPT-quadric}, \cite{auel-pir}, \cite{schreieder1}; higher-dimensional variants are also in \cite[Sect. 3]{schreieder-recent}. 
The presence of nontrivial unramified cohomology in $X$, 
together with a verification of $CH_0$-triviality of a resolution of singularities of $X$, 
show that very general members of those families fail stable rationality.

We write the double cover $T\to S$ as
\[ T:\ \ w^2=x^2+y^2+z^2-2xy-2xz-2yz, \]
a nonsingular quadric surface.
The Brauer group element is
\[ \beta=(xz^{-1},yz^{-1})\in \Br(\C(T)). \]
The quadric surface bundle is, however, not a simple involution surface bundle.
Indeed, the fibers generically along any coordinate line in $\PP^2$ are not
of any of the four permitted degeneration types.

According to the construction in \cite[Thm.\ 10]{KTinvsurf} of a
simple involution surface bundle, we should blow up $S$ as needed so that
the locus
\[
D_{xyz}\cup D_{xzy}\cup D_{yxz}\cup D_{yzx}\cup D_{zxy}\cup D_{zyx},\quad
D_{xyz}: x=0, w=y-z,\ \text{etc.}
\]
on $T$ where $\beta$ is ramified is
a simple normal crossing divisor, which additionally has normal crossings
with the ramification locus
\[ w=0 \]
of $T\to S$.
The first of these conditions is satisfied, but the additional condition
fails since pairs of divisors such as $D_{xyz}$ and $D_{xzy}$ intersect
at points with $w=0$.

Blowing up $S$ at the points $(0:1:1)$, $(1:0:1)$, $(1:1:0)$ yields
exceptional divisors $D_x$, $D_y$, $D_z$.
When we do this, $T$ transforms to a singular surface, whose resolution
requires blowing at $3$ more points to obtain
$\widetilde{S}$ with $3$ more exceptional divisors
$E_x$, $E_y$, $E_z$.
The degree $2$ cover $\widetilde{T}$ is nonsingular, with covering map
\[ \tilde\psi\colon \widetilde{T}\to \widetilde{S} \]
branched over
\[ C'\cup D'_x\cup D'_y\cup D'_z, \]
where primes denote proper transforms.
The locus on $\widetilde{T}$ where $\beta$ ramifies is
\begin{equation}
\label{eqn.xyz}
D'_{xyz}\cup D'_{xzy}\cup D'_{yxz}\cup D'_{yzx}\cup D'_{zxy}\cup D'_{zyx}.
\end{equation}
On $\widetilde{S}$, the Type I locus is $C'\cup D'_x\cup D'_y\cup D'_z$,
and the Type II locus consists of the proper transforms of the
coordinate axes, which all split in $\widetilde{T}$;
see Figure \ref{fig1}.

\begin{figure}
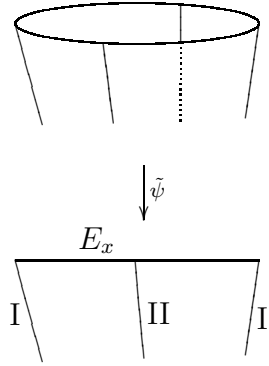

\[
\xy /r9mm/:
  (0.0,3.5)="c";
"c"+(1.8,0.0)*
\xycircle(1.8,0.3){-};
"c"+(0.0,0.0);"c"+(0.4,-1.5)**@{-};
"c"+(3.6,0,0);"c"+(3.4,-1.4)**@{-};
"c"+(1.298,-0.293);"c"+(1.45,-1.465)**@{-};
"c"+(2.4464,-0.2864);"c"+(2.44,-1.432)**@{.};
"c"+(2.4496,0.2864);"c"+(2.4464,-0.2864)**@{-};
"c"+(1.9,-2.1);"c"+(1.9,-2.9)**@{-} ?>*\dir{>};
"c"+(2.1,-2.45)*+{\scriptstyle\tilde\psi};
  (0.0,0.0)="c";
"c"+(0.0,0.0);"c"+(3.6,0.0)**@{-};
"c"+(0.0,0.0);"c"+(0.4,-1.5)**@{-};
"c"+(3.6,0,0);"c"+(3.4,-1.4)**@{-};
"c"+(1.77,0.0);"c"+(1.9,-1.45)**@{-};
"c"+(0.0,-0.8)*+{\mathrm{I}};
"c"+(3.65,-0.85)*+{\mathrm{I}};
"c"+(2.12,-0.78)*+{\mathrm{II}};
"c"+(1.2,0.3)*+{E_x};
\endxy
\]
\caption{Graphical representation of
covering $\tilde\psi\colon \widetilde{T}\to \widetilde{S}$ near
$E_x$, which meets $C'$ and $D'_x$ (both of Type I) and
the proper transform of the coordinate axis $x=0$ (Type II).}
\label{fig1}
\end{figure}

Since every intersection of components in \eqref{eqn.xyz}
is a branch point of the covers which describe the ramification of
$\beta$, there exist standard conic bundles over $\widetilde{T}$
with Brauer class $\beta$ and, correspondingly,
simple involution surface bundles over $\widetilde{S}$.
We may use any such involution surface bundle for the computation of the
unramified Brauer group of $X$
via Theorem \ref{thm.rationalsurfacecase}.

We have
\[ \mathcal{S}=0,\qquad \mathcal{P}=\F_2^6,\qquad
\mathcal{Q}=\F_2^3,\qquad \mathcal{R}=\F_2^3. \]
Ordering the basis of $\mathcal{P}$ as in \eqref{eqn.xyz}, we have image
of $\mathcal{Q}$ spanned by
\[ (1,1,0,0,0,0),\qquad (0,0,1,1,0,0),\qquad (0,0,0,0,1,1), \]
and matrix representation
\[
\begin{pmatrix}
0&0&1&1&1&1\\
1&1&0&0&1&1\\
1&1&1&1&0&0
\end{pmatrix}
\]
of $\mathcal{P}\to \mathcal{R}$, with basis
$(1:0:0)$, $(0:1:0)$, $(0:0:1)$ of $\mathcal{R}$.

Now
\[ \ker(\mathcal{P}/\mathcal{Q}\to \mathcal{R})\cong \F_2, \]
generated by, e.g., $(1,0,1,0,1,0)$, which corresponds to a Brauer group
element that is ramified along each of the three coordinate axes:
\[ (xz^{-1},yz^{-1})\in \Br(\C(S)). \]

Key to this example is the presence of components of the
Type II locus that split in the double cover.
If we start with a quadric surface bundle and hence
$\beta\in \C(T)$ with $\cores_{L/K}(\beta)=0$,
then $\mathcal{P}/\mathcal{Q}=0$ unless some
Type II component splits.
It is not essential, however, to have singular Type II locus.
For instance, there exist nonsingular cubic and quartic curves
in $S=\PP^2$ that meet with tangency at $6$ points lying on a conic.
If we let the quartic curve determine $T$, then the pre-image in $T$ of the
cubic curve has two irreducible components.
We take $\beta\in \Br(\C(T))$ to be the restriction of the
class in $\Br(\C(S))$ determined by a nontrivial unramified degree $2$ cover
of the cubic curve and $X\to S$ a corresponding quadric surface bundle.
In a manner analogous to that described above,
$\widetilde{S}$ and $\widetilde{T}$ may be obtained
by blowing up $6$ points on $S$ and again $6$ points.
Over $\widetilde{S}$ there
is a model of $X$ which is a simple involution surface bundle with
disjoint smooth Type I and II loci.
Theorem \ref{thm.rationalsurfacecase} yields $H^2_{nr}(\C(X))\cong \Z/2\Z$.


\bibliographystyle{plain}
\bibliography{pn}

\begin{thebibliography}{10}

\bibitem{AM}
M.~Artin and D.~Mumford.
\newblock Some elementary examples of unirational varieties which are not
  rational.
\newblock {\em Proc. London Math. Soc. (3)}, 25:75--95, 1972.

\bibitem{abp}
A.~Auel, Chr. B\"ohning, and A.~Pirutka.
\newblock Stable rationality of quadric and cubic surface bundle fourfolds.
\newblock {\em European J. of Math.}, to appear.

\bibitem{auel-pir}
A.~Auel, Chr. B\"ohning, H.-Chr. von Bothmer, and A.~Pirutka.
\newblock Conic bundles over threefolds with nontrivial unramified {B}rauer
  group, 2016.
\newblock {\tt arXiv:1610.04995}.

\bibitem{beau-survey}
A.~Beauville.
\newblock The {L}\"uroth problem.
\newblock In {\em Rationality Problems in Algebraic Geometry}, volume 2172 of
  {\em Lecture Notes in Math.}, pages 1--27. Springer, Cham, 2016.

\bibitem{beauvillesexticdouble}
A.~Beauville.
\newblock A very general sextic double solid is not stably rational.
\newblock {\em Bull. Lond. Math. Soc.}, 48(2):321--324, 2016.

\bibitem{BCTSS}
A.~Beauville, J.-L. Colliot-Th\'el\`ene, J.-J. Sansuc, and P.~Swinnerton-Dyer.
\newblock Vari\'et\'es stablement rationnelles non rationnelles.
\newblock {\em Ann. of Math. (2)}, 121(2):283--318, 1985.

\bibitem{cheltsov-survey}
I.~A. Cheltsov.
\newblock Birationally rigid {F}ano varieties.
\newblock {\em Uspekhi Mat. Nauk}, 60(5):71--160, 2005.

\bibitem{CG}
C.~H. Clemens and Ph.~A. Griffiths.
\newblock The intermediate {J}acobian of the cubic threefold.
\newblock {\em Ann. of Math. (2)}, 95:281--356, 1972.

\bibitem{CTKM}
J.-L. Colliot-Th{\'e}l{\`e}ne, N.~A. Karpenko, and A.~S. Merkurjev.
\newblock Rational surfaces and the canonical dimension of the group {${\rm
  PGL}_6$}.
\newblock {\em Algebra i Analiz}, 19(5):159--178, 2007.

\bibitem{CTO}
J.-L. Colliot-Th{\'e}l{\`e}ne and M.~Ojanguren.
\newblock Vari\'et\'es unirationnelles non rationnelles: au-del\`a de l'exemple
  d'{A}rtin et {M}umford.
\newblock {\em Invent. Math.}, 97(1):141--158, 1989.

\bibitem{ct-pir-cyclic}
J.-L. Colliot-Th{\'e}l{\`e}ne and A.~Pirutka.
\newblock Cyclic covers that are not stably rational.
\newblock {\em Izv. Ross. Akad. Nauk Ser. Mat.}, 80(4):35--48, 2016.

\bibitem{CTP}
J.-L. Colliot-Th{\'e}l{\`e}ne and A.~Pirutka.
\newblock Hypersurfaces quartiques de dimension 3: non-rationalit\'e stable.
\newblock {\em Ann. Sci. \'Ecole Norm. Sup. (4)}, 49(2):371--397, 2016.

\bibitem{EGAIV}
A.~Grothendieck.
\newblock \'{E}l\'ements de g\'eom\'etrie alg\'ebrique. {IV}. \'{E}tude locale
  des sch\'emas et des morphismes de sch\'emas.
\newblock {\em Inst. Hautes \'Etudes Sci. Publ. Math.}, 20, 24, 28, 32,
  1964--67.

\bibitem{HKTconic}
B.~Hassett, A.~Kresch, and Yu. Tschinkel.
\newblock Stable rationality and conic bundles.
\newblock {\em Math. Ann.}, 365(3-4):1201--1217, 2016.

\bibitem{HKTthreefolds}
B.~Hassett, A.~Kresch, and Yu. Tschinkel.
\newblock Stable rationality in smooth families of threefolds, 2018.
\newblock {\tt arXiv:1802.06107}.

\bibitem{HPTdouble}
B.~Hassett, A.~Pirutka, and Yu. Tschinkel.
\newblock A very general quartic double fourfold is not stably rational, 2016.
\newblock {\tt arXiv:1605.03220}, to appear in Algebraic Geometry.

\bibitem{HPT-quadric}
B.~Hassett, A.~Pirutka, and Yu. Tschinkel.
\newblock Intersections of three quadrics in $\mathbb{P}^7$, 2017.
\newblock {\tt arXiv:1706.01371}, to appear in {\it Surveys in Differential
  Geometry}.

\bibitem{HPT}
B.~Hassett, A.~Pirutka, and Yu. Tschinkel.
\newblock Stable rationality of quadric surface bundles over surfaces.
\newblock {\em Acta Mathematica}, 220:341--365, 2018.

\bibitem{HTfano}
B.~Hassett and Yu. Tschinkel.
\newblock On stable rationality of {F}ano threefolds and del {P}ezzo
  fibrations, 2016.
\newblock {\tt arXiv:1601.07074}, to appear in {\it J. Reine Angew. Math.}

\bibitem{MI}
V.~A. Iskovskih and Yu.~I. Manin.
\newblock Three-dimensional quartics and counterexamples to the {L}\"uroth
  problem.
\newblock {\em Mat. Sb. (N.S.)}, 86(128):140--166, 1971.

\bibitem{KT}
M.~Kontsevich and Yu. Tschinkel.
\newblock Specialization of birational types, 2017.
\newblock {\tt arXiv:1708.05699}.

\bibitem{KTsurf}
A.~Kresch and Yu. Tschinkel.
\newblock Models of {B}rauer-{S}everi surface bundles, 2017.
\newblock {\tt arXiv:1708.06277}.

\bibitem{KTinvsurf}
A.~Kresch and Yu. Tschinkel.
\newblock Involution surface bundles over surfaces, 2018.
\newblock {\tt arXiv:1807.01592}.

\bibitem{krylovokada}
I.~Krylov and T.~Okada.
\newblock Stable rationality of del {P}ezzo fibrations of low degree over
  projective spaces, 2017.
\newblock {\tt arXiv:1701.08372}.

\bibitem{manin-ts}
Yu.~I. Manin and M.~A. Tsfasman.
\newblock Rational varieties: algebra, geometry, arithmetic.
\newblock {\em Uspekhi Mat. Nauk}, 41(2):43--94, 1986.

\bibitem{NSh}
J.~Nicaise and E.~Shinder.
\newblock The motivic nearby fiber and degeneration of stable rationality,
  2017.
\newblock {\tt arXiv:1708.02790}.

\bibitem{peyre-ast}
E.~Peyre.
\newblock Progr\`es en irrationalit\'e [d'apr\`es {C}. {V}oisin, {J}.-{L}.
  {C}olliot-{T}h\'el\`ene, {B}. {H}assett, {A}. {K}resch, {A}. {P}irutka, {Y}.
  {T}schinkel et al.], 2016.
\newblock S\'eminaire N. Bourbaki.

\bibitem{pirutkaunramified}
A.~Pirutka.
\newblock Varieties that are not stably rational, zero-cycles and unramified
  cohomology.
\newblock In {\em Algebraic geometry---{S}alt {L}ake {C}ity 2015. {P}art 2},
  volume~97 of {\em Proc. Sympos. Pure Math.}, pages 459--483. Amer. Math.
  Soc., Providence, RI, 2018.

\bibitem{pukh-book}
A.~V. Pukhlikov.
\newblock {\em Birationally rigid varieties}, volume 190 of {\em Mathematical
  Surveys and Monographs}.
\newblock American Mathematical Society, Providence, RI, 2013.

\bibitem{schreieder2}
S.~Schreieder.
\newblock On the rationality problem for quadric bundles, 2017.
\newblock {\tt arXiv:1706.01356}.

\bibitem{schreieder1}
S.~Schreieder.
\newblock Quadric surface bundles over surfaces and stable rationality.
\newblock {\em Algebra Number Theory}, 12(2):479--490, 2018.

\bibitem{schreieder-recent}
S.~Schreieder.
\newblock Stably irrational hypersurfaces of small slopes, 2018.
\newblock {\tt arXiv:1801.05397}.

\bibitem{totaro-JAMS}
B.~Totaro.
\newblock Hypersurfaces that are not stably rational.
\newblock {\em J. Amer. Math. Soc.}, 29(3):883--891, 2016.

\bibitem{voisin}
C.~Voisin.
\newblock Unirational threefolds with no universal codimension {$2$} cycle.
\newblock {\em Invent. Math.}, 201(1):207--237, 2015.

\bibitem{Voisin-survey}
C.~Voisin.
\newblock Stable birational invariants and the {L}\"uroth problem.
\newblock In {\em Surveys in differential geometry 2016. {A}dvances in geometry
  and mathematical physics}, volume~21 of {\em Surv. Differ. Geom.}, pages
  313--342. Int. Press, Somerville, MA, 2016.

\bibitem{zaharia}
A.~Zaharia.
\newblock Characterizations of simple isolated line singularities.
\newblock {\em Canad. Math. Bull.}, 42(4):499--506, 1999.

\end{thebibliography}

\end{document}